\documentclass[journal,10pt]{elsarticle}
\usepackage[OT1]{fontenc}  	
\usepackage[utf8]{inputenc}
\usepackage[greek,swedish,spanish,french,english,USenglish]{babel} 		
\usepackage{amssymb}  			
\usepackage{amsmath}	
\usepackage{amsthm}	
\usepackage{mathrsfs}	
\usepackage{esint}
\usepackage{units}
\usepackage{bbm}   				
\usepackage{color,graphicx} 			
\usepackage[small]{caption}
\usepackage{cases}							
\usepackage{mathpazo}  

\newcommand{\be}{\begin{equation}}
\newcommand{\ee}{\end{equation}}

\renewcommand{\Re}{\mathop{\rm Re}\nolimits}
\newcommand{\sh}{\mathop{\rm sh}\nolimits}
\newcommand{\ch}{\mathop{\rm ch}\nolimits}

\newcommand{\res}{\mathop{\rm res}\limits}
\newtheorem{theorem}{Theorem}
\newtheorem*{corollary}{Corollary}

\makeatletter
\newcommand{\specialnumber}[1]{
\def\tagform@##1{\maketag@@@{(\ignorespaces##1\unskip\@@italiccorr#1)}}}
\makeatother

\makeatletter
\def\ps@pprintTitle{%
\let\@oddhead\@empty
\let\@evenhead\@empty
\let\@oddfoot\@empty
\let\@evenfoot\@oddfoot
}\makeatother

\begin{document}

\begin{frontmatter}

\title{A complement to a recent paper on some infinite sums \\ with the zeta values}

\author{Iaroslav V.~Blagouchine\corref{cor1}} 
\ead{iaroslav.blagouchine@univ-tln.fr}

\begin{abstract}
Recently, several new results related to the evaluation of the series $\sum (-1)^n\zeta(n)/(n+k)$
were published. In this short note we show that this series also possesses
an interesting connection to the values of the $\zeta$--function on the 
critical line and to the Euler constant.
\end{abstract}


\end{frontmatter}

\section{Introduction}
In a recent paper Coppo \cite{coppo_02} investigated the series 
\be\label{73bc27cveywue}
\nu_k \equiv \sum_{j=2}^{\infty} (-1)^j \frac{\zeta(j)}{j+k}\,,\qquad k\in\{-1,0\}\cup\mathbbm{N}\,,
\ee
and obtained various interesting properties by comparing different closed--form expressions
for it. The same series was earlier studied in \cite[p.~413, Eq.~(38)]{iaroslav_08},
where we also obtained a closed--form expression for it.\footnote{Except for the case $k=-1$, which was not studied 
there.} In this short note, we devise yet another expression for the same series, 
showing that there exist an intimate connection between the values of the $\zeta$--function on the 
critical line $\,\nicefrac12+it$, $t\in\mathbbm{R}$, Euler's constant $\gamma$
and the fundamental values of the zeta function at positive integers $\zeta(n)$, $n=2,3,4,\ldots$ 
Furthermore, the obtained expression may also be useful
in that sense that it also holds for non--integer and even complex values of $k$.

\section{The results}
We present our results in the form of two theorems with a corollary. Since the proofs of both theorems
are quite similar, for the purpose of brevity we provide the proof only for the first theorem.

\begin{theorem}
The series $\nu_k$ is closely connected to the values of the $\zeta$--function on the critical line.  
In particular, $\nu_k$ may be evaluated via the following integral with an exponentially decreasing kernel:  
\be\label{p084nbc34p}
\nu_\omega = - \frac{\,1\,}{\,(\omega+1)^2} + \frac{\gamma}{\omega+1}  \, -
\, \frac{1}{2}\!\int\limits_{-\infty}^{+\infty} \! \frac{\zeta\big(\nicefrac{1}{2} \pm ix\big)}
{\big(\nicefrac{1}{2} \pm ix+ \omega \big)\ch\pi x}\, dx\,,\qquad \Re\omega > -\tfrac12
\ee
\end{theorem}

\begin{theorem}
For any $\omega\in\mathbbm{C}$ such that $\Re\omega > -3/2$, we also have
\be\label{lkcje2pin}
\nu_\omega = \frac{1}{2}\!\int\limits_{-\infty}^{+\infty} \! \frac{\zeta\big(\nicefrac{3}{2} \pm ix\big)}
{\big(\nicefrac{3}{2} \pm ix+ \omega \big)\ch\pi x}\, dx\,.
\ee
\end{theorem}

\begin{corollary}
Putting $\omega=k$, $k\in\mathbbm{N}_0$, in previous theorems and comparing \eqref{73bc27cveywue} to the 
expressions obtained in \cite[p.~413, Eq.~(38)]{iaroslav_08}, \cite{coppo_02} and \cite{iaroslav_12}, 
enable us to evaluate the integrals in \eqref{p084nbc34p}
and \eqref{lkcje2pin} in a closed--form. For instance,
\begin{align*}
\int\limits_{-\infty}^{+\infty} \! \frac{\zeta\big(\nicefrac{1}{2} \pm ix\big)}
{\big(\nicefrac{1}{2} \pm ix \big)\ch\pi x}\, dx & = -2 \,,\\
\int\limits_{-\infty}^{+\infty} \frac{\zeta(\nicefrac{1}{2}+ix)}{(\nicefrac{3}{2}+ix)\ch\pi x}\, dx & = \ln2\pi -\frac52\,, \\
\int\limits_{-\infty}^{+\infty} \frac{\zeta(\nicefrac{1}{2}+ix)}{(\nicefrac{5}{2}+ix)\ch\pi x}\, dx & =  \ln2\pi - 4\ln A - \frac{11}{9}\,,\\
\int\limits_{-\infty}^{+\infty} \! \frac{\zeta\big(\nicefrac{3}{2} \pm ix\big)}
{\big(\nicefrac{3}{2} \pm ix \big)\ch\pi x}\, dx & = 2\gamma \,, \\
\int\limits_{-\infty}^{+\infty} \! \frac{\zeta\big(\nicefrac{3}{2} \pm ix\big)}
{\big(\nicefrac{1}{2} \pm ix \big)\ch\pi x}\, dx & = \,
2\!\int\limits_{0}^{1} \! \frac{\Psi(x+1)+\gamma}{x}\, dx\,
= \, 2\kappa_1 +\frac{\pi^2}{6}-\gamma^2-2\gamma_1 \,,
\end{align*}
where $A \equiv e^{\frac{1}{12} -  \zeta^{\prime} (-1)}=1.282427129\ldots$ is the Glaisher-Kinkelin constant,
$\Psi$ is the digamma function (the logarithmic derivative of the $\Gamma$--function), 
$\gamma_1=-0.07281584548\ldots$ is the first Stieltjes constant and \mbox{$\kappa_1=0.5290529699\ldots$} is 
a constant related to Gregory's coefficients $G_n$, see \cite[Appendix]{iaroslav_12}, as well as the sequence A270859 from the OEIS
for more digits of $\kappa_1$.
We do not know
if such a ``simple'' result for these integrals could be found for algebraic or just rational values of $\omega$.
\end{corollary}

% \begin{corollary}
% ... the latter expression over $\omega$ permits one to evaluate another series:
% \end{corollary}

\begin{proof}
Consider the following line integral taken along a contour $C$ consisting of the interval $[-R,+R]$, $R\in\mathbbm{N}$,
on the real axis, and a semicircle of the radius $R$ in the upper half-plane, denoted $C_R$,
\be\label{0923un}
\ointctrclockwise\limits_{C} \frac{\zeta\big(\nicefrac{1}{2} - iz\big)}
{\big(\nicefrac{1}{2} - iz+ \omega \big)\ch\pi z}\, \,dz\,=
\int\limits_{-R}^{+R} \! \frac{\zeta\big(\nicefrac{1}{2} - ix\big)}
{\big(\nicefrac{1}{2} - ix+ \omega \big)\ch\pi x}\, dx \, +  \int\limits_{C_R} \!
\frac{\zeta\big(\nicefrac{1}{2} - iz\big)}
{\big(\nicefrac{1}{2} - iz+ \omega \big)\ch\pi z}\,dz \,,
\ee
with $\omega\in\mathbbm{C}$, $\Re\omega>-\tfrac12$.
On the contour $C_R$ the last integral may be bounded as follows:
\begin{eqnarray}\notag
&& \left|\,  \int\limits_{C_R} \!
\frac{\zeta\big(\nicefrac{1}{2} - iz\big)}
{\big(\nicefrac{1}{2} - iz+ \omega \big)\ch\pi z}\,dz  \, \right|   \,=\, R
\left|\, \int\limits_0^{\pi} \!
\frac{\, \zeta\big(\nicefrac{1}{2} - iRe^{i\varphi}\big)  \, e^{i\varphi}\,}
{\,\big(\nicefrac{1}{2} - iRe^{i\varphi}+ \omega \big) \ch \big(\pi Re^{i\varphi}\big) \,}\,d\varphi \, \right|  \\[2mm]
&& \qquad
\,\leqslant\,R \! \max_{\varphi\in[0,\pi]} \!
\left| \frac{\, \zeta\big(\nicefrac{1}{2} - iRe^{i\varphi}\big)\,}
{\,\nicefrac{1}{2} - iRe^{i\varphi}+ \omega \,} \right| \cdot I_R
\,\leqslant \max_{\varphi\in[0,\pi]} \!
\left| \zeta\big(\nicefrac{1}{2} - iRe^{i\varphi}\big)\right| \cdot I_R \label{98ydfn2k}
\end{eqnarray}
where we denoted 
\be\notag
I_R\,\equiv \int\limits_0^{\pi} \!
\frac{\,d\varphi \,}{\,\big|\ch\! \big(\pi Re^{i\varphi}\big) \big|\,}\,,\qquad R>0\,,
\ee
for brevity. 
Now, in the half--plane $\sigma>1$, the absolute value of $\zeta(\sigma+it)$ may be always 
bounded by a constant $C=\zeta(\sigma)$, which decreases and tends to $1$ as $\sigma\to\infty$. 
In contrast, in the strip $0\leqslant\sigma\leqslant1$ the function $\big|\zeta(\sigma+it)\big|$ is unbounded;
presently, the rate of grow is still not known, but it follows from the general theory of Dirichlet series
that it cannot be faster than $O\big(|t|^{1-\sigma}\big)$, $|t|\geqslant\frac{1}{2}$, in the strip $\frac12\leqslant\sigma<1$, 
see \cite[Theorem 35, pp.~99--102]{chudakov_01_eng}.\footnote{There exist, of course, more sharp estimations, 
such as Huxley's estimations or \emph{Lindel\"of hypothesis},
but we do not need them for our proof (see, for more details, e.g.~\cite[Chapt.~XIII]{titchmarsh_02}, 
\cite{edwards_01}, \cite{voronin_01}, \cite{ivic_01}, \cite{huxley_02}, \cite{lifshits_01}).}
Hence, since $\,\sin\varphi\geqslant0\,$ 
and if $R$ is large enough, this rough estimate gives us 
\be\notag
\left| \zeta\big(\nicefrac{1}{2} - iRe^{i\varphi}\big)\right|\,=\,
\left| \zeta\big(\nicefrac{1}{2} + R\sin\varphi - iR\cos\varphi\big)\right|\,
=\,O\!\left(\!\sqrt{R\,}\right)
\ee
in the interval $\varphi\in[0,\pi].$
On the other hand, as $R$ tends to infinity and remains integer the integral $I_R$ tends to zero 
as $O(1/R)$. To show this, we first note that 
\be\notag
\frac{\,1\,}{\,\big|\ch\! \big(\pi Re^{i\varphi}\big) \big|\,}\,=
\,\frac{\sqrt{2}}{\,\sqrt{\ch(2\pi R\cos\varphi) + \cos(2\pi R\sin\varphi)\,}}\,=\,
O\big(e^{-\pi R|\cos\varphi|}\big ) \,,\qquad R\to\infty\,,
\ee
because $\,0\leqslant\varphi\leqslant\pi\,$ and $R\in\mathbbm{N}$.
Since $\,\big|\ch \! \big(\pi Re^{i\varphi}\big)\big|^{-1}\,$ is symmetric about $\varphi=\frac{1}{2}\pi$,
we may write 
\begin{eqnarray}
\notag
I_R && 
% =\int\limits_0^\pi \!
% \frac{\,\sqrt{2} \, \,}{\,\sqrt{\ch(2\pi R\cos\varphi) + \cos(2\pi R\sin\varphi)\,}} \, d\varphi\\[2mm]
% && \notag
=\int\limits_0^{\frac{\pi}{2}} \!
\frac{\,2\sqrt{2} \,}{\,\sqrt{\ch(2\pi R\cos\varphi) + \cos(2\pi R\sin\varphi)\,}} \, d\varphi\\[2mm]
&&  \label{093ujfo3pjf}
= \, O\!\left(\!\int\limits_0^{\frac{\pi}{2}} \! e^{-\pi R\cos\varphi} \, d\varphi  \right) 
= \, O\!\left(\!\int\limits_0^{\frac{\pi}{2}} \! e^{-\pi R\sin\vartheta} \, d\vartheta  \right)\,,
\qquad R\to\infty\,.
\end{eqnarray}
Now, from the well--known inequality
\be\notag
\frac{2\vartheta}{\pi}\leqslant\sin\vartheta\leqslant\vartheta\,,\qquad \vartheta\in\big[0,\tfrac{1}{2}\pi\big]
\ee
we deduce that
\be\label{987y4gb7}
\frac{\,1-e^{-\frac12\pi^2 R}\,}{\pi R}\leqslant
\int\limits_{0}^{\;\frac{\pi}{2}} \!\! e^{-\pi R\sin\vartheta} \, d\vartheta
\leqslant\frac{1-e^{-\pi R}}{2 R}\,,
\ee
whence $I_R=O(1/R)$ at $R\to\infty$.
Inserting both latter results into \eqref{98ydfn2k}, we obtain 
\be\notag
\left|\,  \int\limits_{C_R} \!
\frac{\zeta\big(\nicefrac{1}{2} - iz\big)}
{\big(\nicefrac{1}{2} - iz+ k \big)\ch\pi z}\,dz  \, \right| =\,O\!\left(R^{\nicefrac{-1}{2}}\right)
\to \, 0 \qquad\text{as} \quad R\to\infty\,, \, R\in\mathbb{N}\,.
\ee
Hence, making $R\to\infty$, equality \eqref{0923un} becomes
\be\label{674hbu74u45g345}
\int\limits_{-\infty}^{+\infty} \! \frac{\zeta\big(\nicefrac{1}{2} - ix\big)}
{\big(\nicefrac{1}{2} - ix+ \omega \big)\ch\pi x}\, dx \,= 
\ointctrclockwise\limits_{C} \frac{\zeta\big(\nicefrac{1}{2} - iz\big)}
{\big(\nicefrac{1}{2} - iz+ \omega \big)\ch\pi z}\, \,dz\,, \qquad \Re\omega>-\tfrac12\,,
\ee
where the latter integral is taken around an infinitely large semicircle in the upper half-plane, denoted $C$. 
The integrand is not a holomorphic function: in $C$ it has the simple poles at $z=z_n\equiv i\left(n-\frac12\right)$,
$n=2,3,4,\ldots$, due to the hyperbolic secant, and a double pole at $z=\frac{i}{2}$, due to both the hyperbolic secant and the 
$\zeta$--function.\footnote{Note that the pole $z=-i(\tfrac12+\omega)$, due to the
denominator $\big(\nicefrac{1}{2} - iz+ \omega \big)$, is located in the lower half-plane.} 
Therefore, by the Cauchy residue theorem 
\begin{eqnarray}\notag
&& \ointctrclockwise\limits_{C} \frac{\zeta\big(\nicefrac{1}{2} - iz\big)}
{\big(\nicefrac{1}{2} - iz+ \omega \big)\ch\pi z}\, \,dz\, = \,2\pi i \left\{
\res_{z=\frac{i}{2}} \! \frac{\zeta\big(\nicefrac{1}{2} - iz\big)}
{\big(\nicefrac{1}{2} - iz+ \omega \big)\ch\pi z} \, + 
\sum_{n=2}^\infty \res_{z=z_n} \! \frac{\zeta\big(\nicefrac{1}{2} - iz\big)}
{\big(\nicefrac{1}{2} - iz+ \omega \big)\ch\pi z} \right\}  \\[3mm]
&& \qquad\qquad\qquad\ \notag
= \,2\pi i \left\{
\lim_{z\to\frac{i}{2}} \! \frac{\partial}{\partial z}\frac{(z-\frac12 i)^2\zeta\big(\nicefrac{1}{2} - iz\big)}
{\big(\nicefrac{1}{2} - iz+ \omega \big)\ch\pi z} \, + 
\sum_{n=2}^\infty \,\left. \! \frac{\zeta\big(\nicefrac{1}{2} - iz\big)}
{\pi \big(\nicefrac{1}{2} - iz+ \omega \big)\sh\pi z} \right|_{z=i(n-\frac12)}\right\}  \\[3mm]
&& \qquad\qquad\qquad \notag
= \,2\pi i \left\{\frac{1}{\pi i}\cdot\frac{\gamma(1+\omega)-1}{(1+\omega)^2} \, + \,
\frac{1}{\pi i} \!\sum_{n=2}^\infty (-1)^{n+1}\frac{\zeta(n)}{n+\omega} \right\}  .
\end{eqnarray}
Equating the left part of \eqref{674hbu74u45g345} with the last result yields \eqref{p084nbc34p}.
\end{proof}

\end{document}